\documentclass[11pt]{amsart}

\usepackage[utf8]{inputenc}
\usepackage[margin=1in]{geometry}
\usepackage{amsmath,amssymb,amsthm}
\usepackage{tikz}
\usetikzlibrary{cd}

\newcommand{\T}{\mathbf{T}}
\newcommand{\Tt}{\tilde{\mathbf{T}}}
\newcommand{\Ter}{\mathcal{T}}
\newcommand{\C}{\mathcal{C}}
\newcommand{\pib}{\mathbf{\pi}}

\newcommand{\tx}{\tilde{x}}
\newcommand{\bSigma}{\mathbf{\Sigma}}
\newcommand{\N}{\mathbb{N}}
\newcommand{\0}{\mathbb{\emptyset}}

\renewcommand{\leq}{\leqslant}

\renewcommand{\geq}{\geqslant}

\newcommand{\SC}{\mathcal{S}}

\newtheorem{theorem}{Theorem}
\newtheorem{lemma}[theorem]{Lemma}
\newtheorem{corollary}[theorem]{Corollary}

\theoremstyle{definition}
\newtheorem*{definition}{Definition}
\newtheorem*{remark}{Remark}

\usepackage{comment}
\usepackage{mathtools}
\usepackage{enumitem}
\usepackage{url}
\usepackage{hyperref}

\title{Borel determinacy: a streamlined proof}
\author{Thomas Buffard, Gabriel Levrel, Sam Mayo}
\date{}

\begin{document}

\begin{abstract}
     First proved my Donald Martin in 1975, the result of Borel determinacy has been the subject of multiple revised proofs. Following Martin's book \cite{martin}, we present a recent streamlined proof which implements ideas of Martin, Moschovakis, and Hurkens. We aim to give a concise presentation that makes this proof approachable to a wider audience.
\end{abstract}

\maketitle


We begin by briefly recalling some definitions and notation used in the rest of the article. We are concerned with \textit{Gale-Stewart games}, i.e. infinite two player games of perfect information where the players, denoted I and II, play alternatively. We refer the reader to \cite{dst} for more on the following definitions, as well as motivation for considering infinite games in descriptive set theory.\\

Given a nonempty set $X$, we let
$$X^{<\N}:=\bigcup_{n\in\N}X^n.$$
A \textit{game tree} $T$ is a nonempty subset of $X^{<\N}$ that is $\subseteq$-closed downward, i.e. for all $p,q \in X^{<\N}$, if $q \in T$ and $p \subseteq q$ then $p \in T$. The elements of $T$ are the \textit{positions} in our game, and we call a position $p \in T$ \textit{terminal} if it has no proper extension in $T$, i.e. there does not exist $q \in T$ such that $p \subsetneq q$. We denote by $|p|$ the length of $p$ considered as a finite sequence. Considering $T$ as a rooted tree, we will occasionally refer to the positions that extend $p$ as the \textit{children of} $p$.

A \textit{move} at a position $p \in T$ is an element $a \in X$ such that $p^\frown \langle a \rangle \in T$, where $^\frown$ denotes the concatenation of two sequences.

We denote by $[T]$ the set of infinite branches through $T$, that is,
$$[T]:=\{x\in X^\N:\text{for all }p\subsetneq x,p\in T\}.$$
A \textit{play} $x$ of our game $T$ is an element of $X^{\leq\N}:=X^{<\N}\cup X^\N$ such that either $x$ is terminal in $T$, or $x\in[T]$. We denote by $\lceil T \rceil$ the set of all plays in $T$, and we call $[T]$ the set of infinite plays. A game tree $T$ is called \textit{pruned} if $T$ has no terminal positions, i.e. $\lceil T \rceil =[T]$.

Given a game tree $T$ and a \textit{payoff set} $A \subseteq \lceil T\rceil$, the associated \textit{game}, denoted $G(A;T)$, is the game played according to the game tree $T$ where a play $x$ is a win for player I if $x \in A$ and a win for player II if $x \in \lceil T\rceil\setminus A$.

A \textit{strategy for player I} (resp. II) in $T$ is a function $\sigma$ which maps a nonterminal position $p \in T$ of even length (resp. odd length) to a move $a$ at $p$. We denote the set of strategies for player I by $\SC_\text{I}(T)$ (resp. $\SC_\text{II}(T)$ for II) and we let $\SC(T) := \SC_\text{I}(T) \cup \SC_\text{II}(T)$ be the set of all strategies.

A position $p \in T$ is \textit{consistent} with a strategy $\sigma$ for I (resp. II) if $p(k) = \sigma(p|_{k})$ for all even (resp. odd) $k < |p|$. A play $x \in \lceil T\rceil$ is \textit{consistent} with $\sigma$ if every position $p \subseteq x$ is consistent with $\sigma$.

A strategy $\sigma$ for player I is \textit{winning} if all plays consistent with $\sigma$ are in the payoff set $A$. Similarly, a strategy $\tau$ for player II is winning if all plays consistent with $\tau$ are not in $A$.

\begin{definition}
A game $G(A;T)$ is \textit{determined} if either player has a winning strategy.
\end{definition}

For any position $p \in T$, we define the \textit{game subtree at $p$} as
\[T_p:= \{q \in T: q \subseteq p \text{ or } p \subseteq q\}.\]
$T_p$ defines a game which is played the same as $T$ but with the first $|p|$ moves fixed.

We equip $\lceil T\rceil$ with the topology whose basic open sets are the sets $\lceil T_p\rceil$ for all $p \in T$ and we say that a game $G(A;T)$ is open, closed, Borel, etc. if $A$ is open, closed, Borel, etc.\\

In \cite{martin}, Martin introduced the notion of a game with taboos. They are the same games defined above, except that all finite plays are declared losing, or \textit{taboo}, for either player, independently of the payoff set $A$.

\begin{definition}
    A game tree $\T$ with taboos is a triple $\langle T, \Ter_\text{I}, \Ter_\text{II} \rangle$ where 
    \begin{itemize}
        \item $T$ is game tree,
        \item $\Ter_\text{I}$ and $\Ter_\text{II}$ form a partition of $\lceil T \rceil \setminus [T]$.
    \end{itemize}
\end{definition}

Positions, moves, plays, and strategies in a game tree with taboos $\T$ are exactly the positions, moves, plays, and strategies in $T$. For a position $p\in T$, the \textit{game subtree} at $p$ is $\T_p:=\langle T_p,\Ter_\text{I}\cap \lceil T_p\rceil,\Ter_\text{II}\cap \lceil T_p\rceil\rangle$. In the rest of this paper, $\T$ will denote an arbitrary game tree with taboos. 

\begin{definition}
    For any game tree with taboos $\T$ and payoff set $A \subseteq [T]$, the associated game with taboos $G(A;\T)$ is the game played according to $T$ where a play $x \in \lceil T\rceil$ is a win for player I if and only if $x \in A$ or $x$ is taboo for II, and a win for II otherwise. That is, $G(A;\T)$ is the same game as $G(A\cup\Ter_\text{II};T)$.
\end{definition}

We equip $[T]$ the subspace topology, considering $[T]$ as a subspace of $\lceil T\rceil$. Since the payoff set $A$ is defined to be a subset of $[T]$, we say that the game $G(A;\T)$ is open, closed, Borel, etc. if $A$ is open, closed, Borel, etc., as a subset of $[T]$.

Games with taboos might seem like a redundant notion, since they model exactly the same games as the ordinary ones we defined first. The point is that games with and without taboos differ topologically, since a priori, $A$ could have a different Borel complexity as a subset of $\lceil T\rceil$ than as a subset of $[T]$. While this turns out not to be the case unless $A$ is open \cite[Lemma 2.1.1]{martin}, the following lemma shows that any potential topological differences between games with or without taboos is irrelevant from the standpoint of determinacy results.
\begin{lemma}\label{levelbylevel}
The determinacy of games with taboos is level-by-level (of the Borel hierarchy) equivalent to the determinacy of infinite games without taboos (i.e. where the game tree is pruned).

\end{lemma}
\begin{proof}
    In one direction, if $G(A;T)$ is an infinite game without taboos, i.e. $\lceil T\rceil=[T]$, then we can consider $T$ as a game tree with taboos, and the resulting game with taboos is completely identical. Hence the determinacy of games with taboos implies that of infinite games without taboos. The other direction is less trivial:
    
    Let $G(A;\T)$ be a game with taboos with Borel payoff set $A\subseteq[T]$. We call a strategy a \textit{taboo-strategy} if every play consistent with it is taboo for the opposite player, and we call a position $p\in T$ \textit{taboo-determined} for player I (resp. II) if I (resp. II) has a taboo-strategy for $\T_p$. Let $T^*$ denote the set of positions in $\T$ that are taboo-determined for either player, and let $T':=T\setminus T^*$.
    
    First, we claim that $T'$ a pruned game tree, i.e. $\lceil T'\rceil = [T']$. Indeed, the children of a taboo-determined position are also taboo-determined, so $T'$ is closed downward under $\subseteq$ and thus a game tree. Suppose toward a contradiction that $p\in\lceil T'\rceil$ is a finite play. Then every position extending $p$ in $T$ must be taboo-determined and so $p$ must be taboo-determined, a contradiction since $p\in T\setminus T^*$.
    
    We'll show that the determinacy of $G(A\cap [T'];T')$ implies that of $G(A;\T)$. Since $A\cap [T']$ is as simple, topologically, as $A$, this will give our sought after level-by-level equivalence.
    
    Let $\sigma' \in \SC_\text{I}(T')$ be a strategy for player I. We use $\sigma'$ to construct to a strategy $\sigma\in\SC_\text{I}(T)$. There are two cases to consider: (a) player II moves to a position $p \in T'$ or (b) player II moves to some $q \in T^*$ for the first time. In (a), $\sigma$ will follow the strategy defined by $\sigma'$. In (b) there are two further sub-cases: if $q$ is taboo-determined for I, then $\sigma$ follows the taboo-strategy defined for $\T_q$. If $q$ is taboo-determined for II, then the parent of $q$ must also be taboo-determined for II. This contradicts the minimality of $q$ in $T^*$, so this sub-case never happens.
    
     It follows that if player I has a winning strategy $\sigma'$ in $G(A\cap [T'];T')$, then $\sigma$ is winning strategy in $G(A;\T)$. Finally, an identical argument shows the same thing for player II.
\end{proof}


Equipped with our definition and basic facts about games with taboos, we proceed with defining a covering. As usual, $\T$ is a fixed game tree with taboos.

\begin{definition}
    A \textit{covering} of $\T$ is a triple $\langle\Tt,\pi,\phi\rangle$, of a game tree with taboos $\Tt$, a \textit{position map} $\pi:\tilde{T}\to T$ such that for all $\tilde{p},\tilde{q}\in\Tt$,
\begin{enumerate}[label=(\roman*)]
    \item $\tilde{p}\subseteq\tilde{q}\rightarrow \pi(\tilde{p})\subseteq\pi(\tilde{q})$,
    \item $|\pi(\tilde{p})|= |\tilde{p}|$,
    \item $\pi(\tilde{p})\in\mathcal{T}_\text{I}\rightarrow\tilde{p}\in\tilde{\mathcal{T}}_\text{I}$,
    \item $\pi(\tilde{p})\in\mathcal{T}_\text{II}\rightarrow\tilde{p}\in\tilde{\mathcal{T}}_\text{II}$,
\end{enumerate}
and a \textit{strategy map} $\phi:\mathcal{S}(\tilde{T})\to\mathcal{S}(T)$ such that for all $\tilde{\sigma} \in \mathcal{S}(\tilde{T})$,
\begin{enumerate}[label=(\roman*)]
    \item $\phi(\tilde{\sigma})$ is a strategy for the same player as $\tilde{\sigma}$,
    \item the restriction of $\phi(\tilde{\sigma})$ to positions of length $< n$ depends only on the restriction of $\tilde{\sigma}$ to positions of length $< n$, for $n \in \N$.
\end{enumerate}
Additionally, we require that $\phi$ satisfies the following \textit{lifting condition}: for every $\tilde{\sigma}\in\mathcal{S}(\tilde{T})$ and $x\in\lceil T\rceil$ consistent with $\phi(\tilde{\sigma})$, there is an $\tilde{x}\in\lceil\tilde{T}\rceil$ such that
\begin{enumerate}
    \item[(i)] $\tx$ is consistent with $\tilde{\sigma}$,
    \item[(ii)] $\pi(\tx)\subseteq x$,
    \item[(iii)] either $\pi(\tx)=x$ or $\tx$ is taboo for the player for whom $\tilde{\sigma}$ is a strategy.
\end{enumerate}
\end{definition}

By \cite[Prop. 2.9(a)]{dst}, the conditions for the position map imply that $\pi$ induces a continuous function $[\tilde{T}]\to[T]$. We will abuse notation and use $\pi$ to refer to this function as well.

We call $\tilde{x}$ as above the \textit{lift of $x$ along $\tilde{\sigma}$}. The importance of this condition is illuminated by Lemma \ref{liftstrat}. In particular, coverings make no reference to a payoff set, but the lifting property ensures that $\phi$ maps winning strategies to winning strategies regardless of what payoff set we choose.

\begin{lemma}\label{liftstrat}
    Let $\C=\langle \Tt,\pi,\phi\rangle$ be a covering of $\T$. For all $A\subseteq[T]$, if $\tilde{\sigma}$ is a winning strategy in $G(\pib^{-1}(A);\Tt)$, then $\sigma:=\phi(\tilde{\sigma})$ is a winning strategy for the same player in $G(A;\T)$.
\end{lemma}

\begin{proof}

Assume that $\tilde{\sigma}$ is a winning strategy for player I (the argument for II is identical). Let $x$ be a play in $T$ consistent with $\sigma$. We must show that $x$ is a win for I in $G(A;\T)$, i.e. either $x$ is taboo for II, or $x\in A$.

Let $\tilde{x}\in\lceil \tilde{T}\rceil$ be the lift of $x$ along $\tilde{\sigma}$. Then $\tx$ is consistent with $\tilde{\sigma}$, so it is a win for I in $G(\pi^{-1}(A);\Tt)$. In particular, $\tx$ is not taboo for I, hence by (iii) of the lifting condition, $\pi(\tx)=x$. Then either $\tx$ is taboo for II, and so $x$ is as well, or $\tx\in\pi^{-1}(A)$, in which case $x\in A$.
\end{proof}

We say that a covering $\langle \Tt,\pi,\phi\rangle$ of $\T$ \textit{unravels} a set $A\subseteq[T]$ if $\pi^{-1}(A)$ is clopen (closed and open) in $[\tilde{T}]$.

\begin{corollary}\label{unraveldet}
    If there is a covering of $\T$ that unravels $A\subseteq[T]$, then $G(A;\T)$ is determined.
\end{corollary}
\begin{proof}
    Follows from Lemma \ref{liftstrat}, and Lemma \ref{levelbylevel} applied to clopen determinacy for infinite games (see e.g. \cite[Theorem 13.17]{dst}).
\end{proof}

\begin{remark}
    If a covering unravels $A\subseteq[T]$, then the same covering unravels the complement $[T]\setminus A$.
\end{remark}

We call a covering $\C=\langle \Tt,\pi,\phi\rangle$ of $\T$ a $k$\textit{-covering of} $\T$ if the first $k$ levels of $T$ and $\tilde{T}$ are the same, and $\pi$, $\phi$ are the identity map up to level $k$.

Let $\C_1=\langle \Tt_1,\pi_1,\phi_1\rangle$ be a $k_1$-covering of $\T_0$ and $\C_2=\langle \Tt_2,\pi_2,\phi_2\rangle$ be a $k_2$-covering of $\T_1$. The composition $\C_1 \circ \; \C_2$ is defined to be the $\min\{k_1, k_2\}$-covering $$\langle \T_2, \pi_1 \circ \pi_2, \phi_1 \circ \phi_2 \rangle.$$

In the proof of Borel determinacy (Theorem \ref{maininduct}) we will be faced with a countable system of coverings $\C_j=\langle \T_j, \pi_{j, i}, \phi_{j, i}\rangle$, such that each subsequent covering unravels a different part of the original game tree. The following lemma constructs a certain inverse limit that lets us simultaneously cover each of these game trees.
\begin{lemma}\label{invlim}
    Let $k\in\N$, and let $((\T_i)_{i\in\N},(\C_{j,i})_{i\leq j\in\N})$ be an inverse system of trees and $k$-coverings, where each $\T_i$ is a game tree with taboos, each $\C_{j,i}=\langle \T_j,\pi_{j,i},\phi_{j,i}\rangle$ is a $k$-covering of $\T_i$, and $\C_{i_3,i_1}=\C_{i_2,i_1}\circ\C_{i_3,i_2}$ for all $i_1\leq i_2\leq i_3$. Suppose moreover that for all $n\in\N$, the first $n$ levels of the $\T_i$'s eventually stabilize, i.e. there exists some $i_n$ such that for all $j'\geq j\geq i_n$, $\C_{j',j}$ is an $n$-covering.
    
    Then there exists a game tree $\T_\infty$ and $k$-coverings $\C_{\infty, i} := \langle \T_{\infty}, \pi_{\infty, i}, \phi_{\infty, i} \rangle$ of $\T_i$ for all $i$, such that $\C_{\infty,i}=\C_{j,i}\circ\C_{\infty,j}$ for all $i\leq j$.
\end{lemma}
\begin{proof}
    Let $_n\T_j$ denote the restriction of $\T_j$ to its first $n$ levels, and let $i_n$ be as above. For all $n$, we set $_n\T_\infty:={_n\T_{i_n}}$. Thus we identify positions (and strategies) up to level $n$ in $\T_\infty$ with those in $\T_{i_n}$. This is well-defined and independent of choice of $i_n$ by assumption.
    
    We define $\pi_{\infty,j}$ and $\phi_{\infty,j}$ by their restrictions to the first $n$ levels of $\T_\infty$. Specifically, for $x\in {_n\T_\infty}$, we set $\pi_{\infty,j}(x):=x$ if $j\geq i_n$, and otherwise set $\pi_{\infty,j}(x)=\pi_{i_n,j}(x)$. Similarly, for a strategy $\sigma$ in $_n\T_\infty$, set $\phi_{\infty,j}(\sigma)=\sigma$ if $j\geq i_n$, and set $\phi_{\infty,j}(\sigma)=\phi_{i_n,j}(\sigma)$ otherwise. Since, by assumption, $i_k=0$, $\pi_{\infty,j}$ and $\phi_{\infty,j}$ both satisfy the conditions of a $k$-covering. It is also clear that $\pi_{\infty,i}=\pi_{j,i}\circ\pi_{\infty,j}$ and $\sigma_{\infty,i}=\sigma_{j,i}\circ\sigma_{\infty,j}$ for all $i\leq j$. It remains to show that $\C_{\infty,j}$ satisfies the lifting condition:
    
    Let $\tilde{\sigma}$ be a strategy in $\T_\infty$, and let $x\in\lceil T_j\rceil$ be consistent with $\phi_{\infty,j}(\tilde{\sigma})$. We define $\tilde{x}$, the lift of $x$ along $\tilde{\sigma}$, by its restriction to the first $n$ levels of $\T_\infty$, for all $n\in\N$. First, if $j\geq i_n$, then $\pi_{\infty,j}|_n$ and $\phi_{\infty,j}|_n$ are the identity, so we can set $\tilde{x}|_n=x|_n$. Next, if $j<i_n$, let $y\in\lceil T_{i_n}\rceil$ be the lift of $x$ along $\phi_{\infty,i_n}(\tilde{\sigma})$, via the covering $\C_{i_n,j}$. Again, $\pi_{\infty,i_n}|_n$ and $\phi_{\infty,i_n}|_n$ are the identity by assumption, so we set $\tx|_n=y|_n$. It now follows from the lifting condition for $\C_{i_n,j}$ that $\tx$ is a valid lift of $x$.
\end{proof}

Note that the condition above that each level of the trees stabilize is precisely the reason we have considered $k$-coverings (and not just coverings).

We now state and prove the main theorem:

\begin{theorem}\label{maininduct}
    Let $\T$ be a game tree with taboos, and let $k\in\N$. If $A\subseteq[T]$ is Borel, then there exists a $k$-covering of $\T$ that unravels $A$.
\end{theorem}
\begin{proof}
    We prove the following by induction on countable ordinals $\alpha\geq1$:\\
    
    \textit{$(\dagger)_\alpha$ For all $A\subseteq[T]$ such that $A\in\bSigma_\alpha^0$, and for all $k\in\N$, there is a $k$-covering of $\T$ that unravels $A$.}\\

    We postpone the proof of $(\dagger)_1$ to Lemma \ref{basecase}. Now let $\alpha>1$ and assume by induction that $(\dagger)_\beta$ holds for all $1\leq\beta<\alpha$. Let $A\subseteq[T]$ with $A\in\bSigma_\alpha^0$. Then $A=\bigcup_{i\in\N}B_i$ where $B_i\in\mathbf{\Pi}_{\beta_i}^0$ for some $\beta_i<\alpha$. We define an inverse system $((\T_i)_{i\in\N},(\C_{j,i})_{i\leq j\in\N})$ of trees and coverings that successively unravels the sets $B_i$:
    
    Let $\T_0:=\T$, with $\C_{0,0}$ the identity. For $n\in\N$ suppose that we have defined $\T_{j}$ and $\C_{j,i}$ for all $i\leq j\leq n$. By $(\dagger)_{\beta_n}$, let $\C_{n+1,n}=\langle\T_{n+1},\pi_{n+1,n},\phi_{n+1,n}\rangle$ be a $(k+n)$-covering of $\T_n$ that unravels $\pi^{-1}_{n,0}(B_n)$. We let $\C_{n+1,n+1}$ be the identity, and for $j<n+1$ we set $\C_{n+1,j}=\C_{n,j}\circ\C_{n+1,n}$.

    By construction, each $\C_{j,i}$ is a $(k+i)$-covering of $\T_i$. Hence the conditions of Lemma \ref{invlim} are satisfied, so we let $\T_\infty$ and $(\C_{\infty,i})_{i\in\N}$ be the inverse limit constructed by that lemma. By continuity, $\pi_{\infty,0}^{-1}(B_n)$ is clopen for all $n$, and thus $\pi_{\infty,0}^{-1}(A)$ is open. Finally, Lemma \ref{basecase} grants us a $k$-covering $\tilde{\C}$ of $\T_\infty$ that unravels $\pi_{\infty,0}^{-1}(A)$, and we conclude that $\C_{\infty,0}\circ\tilde{\C}$ is a $k$-covering of $\T$ that unravels $A$.
\end{proof}
From Theorem \ref{maininduct} and Corollary \ref{unraveldet} our main result follows:
\begin{corollary}[Borel Determinacy]
If $A\subseteq[T]$ is Borel, then $G(A;\T)$ is determined.
\end{corollary}

It remains to prove the base case of the induction in the proof of Theorem \ref{maininduct}. This is the key technical step in the proof of Borel determinacy, and it is certainly the hardest to motivate. This may be surprising on a first pass, since simply proving open or closed determinacy is fairly elementary. The difficulty arises from the lifting condition, which requires a substantial amount of care to satisfy.

\begin{lemma}\label{basecase}
    If $A\subseteq[T]$ is open or closed and $k\in\N$, then there is a $k$-covering of $\T$ that unravels $A$.
\end{lemma}
\begin{proof}
    We define a $k$-covering $\C:=\langle\Tt,\pi,\phi\rangle$ of $\T$ and show that it unravels $A$. It suffices to assume that $A$ is closed since any cover that unravels $A$ also unravels its complement. Increasing $k$ if necessary, we also assume that $k$ is even.
    
    First, we define $\Tt$. Since $\C$ should be a $k$-covering, we make the first $k$ levels of $\Tt$ identical to those of $\T$, including taboos. As $k$ is even, player I makes the $(k+1)$'th move. Let $p\in\Tt$ with $|p|=k$, so we identify $p$ with a position in $\T$. Position $p$ is terminal (hence taboo) in $\Tt$ if and only if $p$ is terminal $\T$, so we assume that $p$ is not terminal. Let $a$ be a move for I in $\T$ at $p$. Let $Z$ be the set of positions $q\in \T$ such that:
    \begin{enumerate}[label=(\roman*)]
        \item $p^\frown\langle a\rangle\subsetneq q$,
        \item $q$ is not terminal in $\T$,
        \item $[\T_q]\cap A=\0$,
        \item $(\forall r)(p^\frown\langle a\rangle\subsetneq r\subsetneq q\rightarrow[\T_r]\cap A\neq\0)$.
    \end{enumerate}
    That is, $q \in Z$ is a non-terminal extension of $p^\frown\langle a\rangle$ such that I can only win in the game subtree at $q$ by reaching a taboo. Condition (iv) states that $q$ is a minimal such position.
    
    Then in $\Tt$, I plays a move of the form
    $$\langle a,X\rangle,$$
    where $a$ is a move for I in $T$ at $p$, and $X$ is a subset of $Z$ (where $Z$ depends on $a$ as above).
    
    The idea is that by playing this move, I is claiming that for every $q\in X$ they have a winning strategy for $G(A;\T_q)$, and conceding that II can win $G(A;\T_q)$ for $q\in Z\setminus X$ (note that I can only win in such a $\T_q$ by reaching a taboo). Hence I is claiming that the game should be over whenever a position in $Z$ is reached, with I the winner if said position is in $X$ and II the winner otherwise.
    
    If $p^\frown\langle a\rangle$ is taboo in $\T$, we declare $p^\frown\langle\langle a,X\rangle\rangle$ to be taboo for the same player in $\Tt$. Otherwise, $p^\frown\langle\langle a,X\rangle\rangle$ is not taboo in $\Tt$, and player II has two options: (1) \textit{accept} or (2) \textit{challenge} I's claim.
    
    Option (1): If II \textit{accepts the set $X$}, they must play a move of the form
    $$\langle1,b\rangle,$$
    where $b$ is a legal move for II in $\T$ at position $p^\frown\langle a\rangle$. The game now proceeds just as in $\T$ from the position $p^\frown\langle a\rangle^\frown\langle b\rangle$, unless a position in $Z$ is reached. Specifically, let $\tilde{q}:=p^\frown\langle\langle a,X\rangle\rangle^\frown\langle\langle 1,b\rangle\rangle^\frown s$ be a position in $\Tt$, so that $q:=p^\frown\langle a\rangle^\frown\langle b\rangle^\frown s$ is the corresponding position in $\T$. In accordance with our explanation above, if $q\in X$, we let $\tilde{q}$ be taboo for II in $\Tt$, otherwise if $q\in Z\setminus X$, we let $\tilde{q}$ be taboo for I.
    
    Option (2): If II challenges I's claim, they want to show that they can with the game $G(A;\T_r)$ for some $r \in X$. Specifically, they play a move of the form
    $$\langle2,r,b\rangle,$$
    where $r\in X$ and $p^\frown\langle a\rangle^\frown\langle b\rangle\in \T_r$. In this case we say that II has \textit{challenged position $r$}. Then the game is played just as in $\T$, but where all the moves must be played in the game subtree $\T_r$.
    
    We can now define $\pi$ as the map sending each position in $\Tt$ to the corresponding one in $\T$ by forgetting the extra information ($1$, $2$, $X$, and $r$) in moves $k+1$ and $k+2$. It is clear that this satisfies the conditions for a position map.
    
    We check that $\pi^{-1}(A)$ is clopen, as is required for unraveling: Let $\tx$ be an infinite play in $\Tt$. If II accepted, then no subsequence of $\pi(\tx)$ belongs to the associated $Z$. Thus $[T_q]\cap A\neq\0$ for all finite $q\subsetneq\pi(\tx)$, but $A$ is closed, so $\pi(\tx)\in A$. On the other hand, if II challenged, then $\pi(\tx)$ extends some position in $Z$, so $\pi(\tx)\not\in A$. Thus $\pi^{-1}(A)$ is exactly the set of infinite plays in $\Tt$ where II accepts, which is clopen as it is decided after move $k+1$.
    
    It remains to define the strategy map $\phi$ and show that it satisfies the lifting property. We first consider a strategy $\tilde{\sigma}$ for player I in $\Tt$. We define $\phi(\tilde{\sigma})$ by describing a play of $\T$ from the point of view of player I. In doing so we omit the definition of the strategy for positions in $\T$ not consistent with $\phi(\tilde{\sigma})$—these can be assigned arbitrarily as long as they satisfy the conditions of a strategy map.
    
    For the first $k$ moves, the games $\T$ and $\Tt$ are identical so I can just follow the moves dictated by $\tilde{\sigma}$ up to this point. If the game reaches a non terminal position of length $k$, $\tilde{\sigma}$ provides a move $a$ for I and a set of positions $X$, and we dictate that I play the move $a$ in $\T$.
    
    From this point, I plays according to $\tilde{\sigma}$ under the assumption that II has accepted the set $X$. This gives I a move to play unless a position in $p\in Z$ is reached, since then the corresponding position in $\Tt$ is terminal, and $\tilde{\sigma}$ does not give a move. Suppose such a position $p\in Z$ has been reached. If $p\in Z\setminus X$, then I follows an arbitrary strategy in $\T_p$ from this point on, since they are effectively conceding a loss here. On the other hand, if $p\in X$, then I follows $\tilde{\sigma}$ according to the assumption that II challenged position $p$, i.e. that II played the move $\langle2,p,b\rangle$ where $b$ is the $(k+2)$'th move of $p$.

    To verify the lifting condition, let $x$ be a play of $\T$ that is consistent with $\phi(\tilde{\sigma})$ as described above. We will define the lift $\tilde{x}$ of $x$ along $\tilde{\sigma}$. If $x$ does not extend any position in $Z$, we let $\tilde{x}$ be the corresponding play in $\Tt$ with the assumption that II accepted. If $x$ extends a position in $p\in Z\setminus X$, we let $\tilde{x}$ be the (taboo for I) position in $\Tt$ corresponding to $p$, under the assumption that II accepted. Finally, if $x$ extends a position in $p\in X$, we let $\tilde{x}$ be the corresponding play in $\Tt$ under the assumption that II challenged $p$. It is easily seen in each case that $\tilde{x}$ is a valid lift.
    
    Now, we do the same for player II by fixing a strategy $\tilde{\tau} \in \SC_\text{II}(\tilde{T})$ and describing a play in $\T$ from II's perspective. Again, II follows $\tilde{\tau}$ for the first $k$ moves.
    
    Suppose then that a nonterminal position $q^\frown\langle a\rangle$ in $\T$ of length $k+1$ has been reached, i.e. I played $a$ for their $(k+1)$'th move. Let $Z$ be the set of positions as above (recall that this set depends on $q$ and $a$). Let $Y$ be the set of positions $r\in Z$ that are never challenged by $\tilde{\tau}$, meaning that for any $X\subseteq Z$, $\tilde{\tau}$ will never respond to the move $\langle a,X\rangle$ by challenging $r$. Then II plays their $(k+2)$'th move in $\T$ according to $\tilde{\tau}$ under the assumption that I played $\langle a,Y\rangle$ for move $k+1$. Note that by definition of $Y$, $\tilde{\tau}$ will always have II accept here.

    From this point, II continues to play according to $\tilde{\tau}$ under the assumption that I has played the set $Y$. This provides II a strategy, unless a position in $Z$ is reached, for then the corresponding position in $\Tt$ is terminal. Suppose a position $p\in Z$ has been reached. If $p\in Y$, then II follows an arbitrary strategy in $\T_p$ from this point on, as they are effectively conceding a loss. If $p\in Z\setminus Y$, then by definition of $Y$ there is a subset $X\subseteq Z$ such that if I plays $X$, then II rejects $p$. Then II follows $\tilde{\tau}$, but now under the assumption that I played the set $X$, and consequently that II challenged the position $p$.

    Finally, we let $x$ be a play of $\T$, consistent with $\phi(\tilde{\tau})$, and define a lift $\tx$ of $x$ along $\tilde{\tau}$. If $x$ does not extend a position in $Z$, we let $\tilde{x}$ be the corresponding position in $\Tt$ with the assumption that I played the set $Y$. If $x$ extends a position $p\in Y$, we let $\tx$ be the (taboo for II) position in $\Tt$ corresponding to $p$, under the assumption that I played $Y$. Finally, if $x$ extends a position $p\in Z\setminus Y$, we let $\tx$ be the corresponding play in $\Tt$ under the assumption that I played the set $X$ as in the previous paragraph (recall that this $X$ depends on $p$), and so II challenged $p$. Again, it is easily seen that these define valid lifts.
\end{proof}

\section*{Acknowledgements.} This report is the result of an undergraduate research project at McGill University, supervised by Anush Tserunyan. We would like to thank Anush for her exceptional teaching and supervision throughout this project.

\bibliographystyle{plain}
\bibliography{refs}

\end{document}